\documentclass[reqno]{amsart}
\usepackage[english]{babel}
\usepackage{amscd,amssymb,amsmath,amsfonts,latexsym,amsthm}
\usepackage{ulem}
\usepackage{inputenc}
\usepackage{graphicx,color}
\newtheorem{thm}{Theorem}[section]
\newtheorem{cor}[thm]{Corollary}
\newtheorem{lem}[thm]{Lemma}
\newtheorem{prop}[thm]{Proposition}
\newtheorem{defn}[thm]{Definition}
\newtheorem{exam}[thm]{Example}
\newtheorem{rem}[thm]{Remark}

\numberwithin{equation}{section}

\newcommand{\ZZ}{\mathbb{Z}}

\def\cal{\mathcal }

\DeclareMathOperator{\codim}{codim}
\DeclareMathOperator{\Span}{Span}
\DeclareMathOperator{\diag}{diag}

\begin{document}

\title[On solvable Lie superalgebras of maximal rank]{On solvable Lie superalgebras of maximal rank}

\author{B.A. Omirov, I.S. Rakhimov, G.O. Solijanova}

\address{Bakhrom A.Omirov \newline \indent
Institute for Advances Study in Mathematics, Harbin Institute of Technologies, Harbin, 150001, China} \email{{\tt omirovb@mail.ru}}

\address{Isamiddin S.Rakhimov \newline \indent Universiti Teknologi MARA (UiTM), Shah Alam, Malaysia} \email{\tt isamiddin@uitm.edu.my}

\address{Gulkhayo O.Solijanova \newline \indent National University of Uzbekistan, Tashkent, Uzbekistan}
\email{{\tt gulhayo.solijonova@mail.ru}}

\begin{abstract} In this paper we establish some basic properties of superderivations of Lie superalgebras. Under certain conditions, for solvable Lie superalgebras with given nilradicals, we give estimates for upper bounds to dimensions of complementary subspaces to the nilradicals.  Moreover, under these conditions we describe the solvable Lie superalgebras of maximal rank. Namely, we prove that an arbitrary solvable Lie superalgebra of maximal rank is isomorphic to the maximal solvable extension of nilradical of maximal rank.

\end{abstract}



\maketitle


\section{Introduction}

Currently the theory of Lie superalgebras is one of the most actively studied branches of the modern algebra and theoretical physics. The basic results on Lie superalgebras theory can be found in \cite{Kac}, \cite{Kac2} and references therein. It is well known that Lie superalgebras are a generalization of Lie algebras and they are important in theoretical physics because of description of supersymmetry. Lie superalgebras are also important in two aspects:
\begin{itemize}
\item Graded Lie algebras have appeared in the mathematical literature in deformation theory. The first basic example of graded Lie algebras was provided by A. Nijenhuis \cite{Nijenhuis} and then by A. Fr\"{o}licher and
A. Nijenhuis in their paper \cite{FrolicherNijenhuis}. The basic role that this
object plays in the theory of deformation of algebraic structures was discovered by M. Gerstenhaber in a fundamental series of papers \cite{Gerstenhaber1,Gerstenhaber2}.

\item From a rather different point of view (motivated primarily by problems of second quantization), the subject
was introduced by F.A. Berezin and G.I. Kac \cite{BerezinKac}.

Lie algebras graded by the two-element additive group $\mathbb{Z}_2$ received a special name,
``Lie superalgebras'', due to the development of the so-called ``supermathematics'' related to some necessities of the quantum mechanics and nuclear physics (see \cite{Berezin1,Berezin2,Berezin3}).
\end{itemize}


Let now to give a brief review on the main results of the theory of finite-dimensional Lie algebras with comparison to Lie superalgebras (see \cite{Kac}).

A Lie superalgebra  $\mathcal{L}$ contains a unique maximal solvable ideal $\cal{R}$ (the solvable radical). The Lie superalgebra $\mathcal{L}/\cal{R}$ is semisimple (that is, has no solvable ideals).
Therefore, the study of finite-dimensional Lie superalgebras is reduced in a certain sense to the studies of semisimple and of solvable Lie superalgebras. However, an analogue of Levi theorem on $\mathcal{L}$ being a semidirect sum of $\cal{R}$ and $\mathcal{L}/\cal{R}$
is not true, in general, for Lie superalgebras.

The crucial result of the theory of finite-dimensional solvable Lie algebras is Lie's theorem, which asserts that every finite-dimensional irreducible representation of a solvable Lie algebra over $\mathbb{C}$ is one-dimensional. For Lie superalgebras this is not true, in general as well.

The well known theorem on semisimple Lie algebras to be a direct sum of simple
ones is by no means true for Lie superalgebras. However, there is a
construction that allows to describe finite-dimensional semisimple Lie
superalgebraa in terms of simple ones (see \cite{Kac} Section 5.1.3, Theorem 6). It is similar
to the construction in \cite{Block}.

The fundamental problem of classifying the finite-dimensional
simple Lie superalgebras over an algebraically closed field of characteristic 0 has been solved in \cite{Kac}. Thus, like in the case of Lie algebras the classification problem of Lie superalgebras is brought to that of the solvable Lie superalgebras.

The main goal of this paper is to study the complex solvable Lie superalgebras of the maximal rank with a given nilradical. As one can notice from the facts mentioned above, the structure theory of Lie superalgebras is more complicated than that of Lie  algebras. In particular, Lie theorem is not true for a solvable  Lie superalgebras. Even, the square of a solvable superalgebra is not necessarily nilpotent \cite{Repovs}, \cite{Wang2002}.

Recall that, a solvable Lie algebra $\cal R$ of the form $\cal N\dot{+}\cal T_{max}$ is called of maximal rank if  $\dim\cal T_{max}=dim (\cal N/\cal N^2)$, where $\cal T_{max}$ is a maximal torus of the nilradical $\cal N$ (see \cite{Meng}). Due to the description obtained in \cite{Khal} we conclude that a complex solvable Lie algebra of maximal rank is equivalent to a solvable Lie algebra that satisfies the condition $\codim \cal N=\dim (\cal N/\cal N^2)$. The uniqueness (up to isomorphism) of such algebra is due to the conjugacy of the maximal tori (see \cite{Mostow}).

In this paper we establish some basic properties of even superderivations of Lie superalgebras. The notion of characteristically nilpotent Lie algebra is extended to the case of Lie superalgebra. Under two conditions that guarantee the recovery of sufficiently enough information on a solvable Lie superalgebra from its nilradical, we get estimates for upper bounds to dimensions of complementary subspaces to nilradicals of solvable Lie superalgebras. We also present an example that shows that the considered conditions are essential for the obtained estimations. Besides,
under the two conditions (square of solvable Lie superalgebra belongs to its nilradical and characterization of applicability of analogue of Lie's theorem) we obtain the description (up to isomorphism) of complex solvable Lie superalgebras of maximal rank. In fact, an arbitrary solvable Lie superalgebra of maximal rank appears to be isomorphic to the maximal solvable extension of nilradical of maximal rank. Using results of \cite{Wang2003} we conclude that any even superderivation of solvable Lie superalgebra of maximal rank is inner. Finally, special types of maximal solvable extensions of nilpotent Lie superalgebras are described.

In this paper all spaces and superalgebras are considered to be finite-dimensional and over the field $\mathbb{C}$.

\section{Preliminaries}

A vector space $V$ is said to be $\ZZ_2${\it -graded} if it admits a decomposition in direct sum, $V=V_{\bar 0} \oplus V_{\bar 1}$, where ${\bar 0}, {\bar 1} \in \ZZ_2$. An element $x \in V$ is called {\it homogeneous of degree}  $|x|$ if it is an element of $V_{|x|}$, $|x|\in \ZZ_2$. In particular, the elements of $V_{\bar 0}$ (resp. $V_{\bar 1}$) are also called {\it even} (resp. {\it odd}).

\begin{defn} [\cite{Kac}] A  Lie superalgebra is a $\ZZ_2$-graded vector space  $\mathcal{L}=\mathcal{L}_{\bar 0} \oplus \mathcal{L}_{\bar 1}$, with an even bilinear commutation operation (or ``supercommutation'') $[\cdot,\cdot]$, which for arbitrary homogeneous elements $x, y, z$ satisfies the conditions

\begin{enumerate}
\item[1.] $[x,y]=-(-1)^{|x| |y|}[y,x],$

\item[2.]   $(-1)^{|x| |z|}[x,[y,z]]+(-1)^{|x| |y|}[y,[z,x]]+(-1)^{|y||z|}[z,[x,y]]=0$ {\it (super Jacobi identity).}

\end{enumerate}
\end{defn}

Thus, $\mathcal{L}_{\bar 0}$ is an ordinary Lie algebra and $\mathcal{L}_{\bar 1}$ is a module over $\mathcal{L}_{\bar 0}$.

Note that by applying the condition 1 the super Jacobi identity
can be transformed to the following one
$$[x,[y,z]]=[[x,y],z]-(-1)^{|y| |z|}[[x,z],y],$$
which is called super Leibniz identity \cite{Albeverio}.

Recall the definition of superderivations of Lie superalgebras \cite{Kac2}. A superderivation $d$ of degree $|d|$ ($|d|\in \ZZ_2$) of a Lie superalgebra $\mathcal{L}$, is an endomorphism $d \in End(\mathcal{L})_{|d|}$ with the property
\begin{align}\label{der}
d([x,y])=[d(x),y] + (-1)^{|d| \cdot |x|} [x, d(y)].
\end{align}

If we denote by $Der(\mathcal{L})_{s}$ the space of all superderivations of degree $s$, then $Der (\mathcal{L})=Der(\mathcal{L})_0\oplus Der(\mathcal{L})_1$ is the Lie superalgebra of superderivations of $\mathcal{L}$ with respect to the bracket $[d_i,d_j]=d_i\circ d_j - (-1)^{|d_i| |d_j|} d_j\circ d_i$, where $Der(\mathcal{L})_0$ consists of even superderivations and $Der(\mathcal{L})_1$ of odd ones. Clearly, $Der(\mathcal{L})_0$ forms a Lie algebra.

Note that for a fixed element $x\in \mathcal{L}$ the operator $ad_x : \mathcal{L} \to \mathcal{L}$ defined by $ad_x(y)=[y,x]$ is a superderivation of the Lie superalgebra $\mathcal{L}$, which is called {\it inner superderivation}. The set of all inner superderivations denoted by $Inder(\mathcal{L})$.

The  {\it descending central sequence} and {\it derived sequence} for a Lie superalgebra $\mathcal{L}=\mathcal{L}_0 \oplus \mathcal{L}_{1}$ are defined similar to the case of Lie algebras:
 $$\mathcal{L}^1=\mathcal{L},\ \mathcal{L}^{k+1}=[\mathcal{L}^k,\mathcal{L}],\ \mbox{and}\ \mathcal{L}^{[1]}=\mathcal{L},\ \mathcal{L}^{[k+1]}=[\mathcal{L}^{[k]}, \mathcal{L}^{[k]}],\quad  k \geq 1, \ \mbox{respectively}.$$

 \begin{defn}
 A Lie superalgebra $\mathcal{L}$ is called nilpotent (respectively, solvable) if there exists  $s\in\mathbb{N}$ (respectively, $k\in\mathbb{N}$) such that $\mathcal{L}^s=0$ (respectively, $\mathcal{L}^{[k]}=0$.)
  \end{defn}

With the descending central sequence for a nilpotent Lie superalgebra $\mathcal{N}$ we associate the natural filtration
\begin{align}\label{filtr}
\mathcal{L}^1\supset \mathcal{L}^{2} \supset \dots \supset  \mathcal{L}^{k} \supset \mathcal{L}^{k+1} \supset \cdots.
\end{align}

Let us define now the following crucial sequences for an arbitrary Lie superalgebra $\mathcal{L}=\mathcal{L}_0\oplus \mathcal{L}_1$:

  $${\cal C}^1(\mathcal{L}_0)=\mathcal{L}_0,\ {\cal C}^{k+1}(\mathcal{L}_0)=[{\cal C}^{k}(\mathcal{L}_0),{\mathcal{L}_0}],\quad
{\cal C}^1(\mathcal{L}_1)=\mathcal{L}_1,\ {\cal C}^{k+1}(\mathcal{L}_1)=[{\cal C}^{k}(\mathcal{L}_1),{\mathcal{L}_0}], \quad  k\geq1.$$

  \begin{thm}\cite{Gilg}\label{Gilg}
The Lie superalgebra $\mathcal{L}=\mathcal{L}_0\oplus \mathcal{L}_1$ is nilpotent if and only if there exist positive integers $p$ and $q$ such that ${\cal C}^p(\mathcal{L}_0)={\cal C}^q(\mathcal{L}_1)=\{0\}.$
  \end{thm}

Note that Engel's theorem remains to be valid for Lie superalgebras too (see \cite{scheunert}).

\begin{thm}[Engel's theorem] \label{thmEngel} A Lie superalgebra $\mathcal{L}$ is nilpotent if and only if $ad_x$ is nilpotent for every homogeneous element $x$ of $\mathcal{L}$.
\end{thm}

Nevertheless, neither analogue of Lie's theorem nor its corollaries for solvable Lie superalgebras are true, in general. However, it is known that the analogue of Lie's theorem for a solvable Lie superalgebra  $\mathcal{R}=\mathcal{R}_0\oplus \mathcal{R}_1$ is true if and only if $[\mathcal{R}_1,\mathcal{R}_1] \subseteq [\mathcal{R}_0,\mathcal{R}_0]$ (see Proposition 5.2.4 in \cite{Kac}). Moreover, the classical result from the theory of Lie algebras on nilpotency of the square of a solvable Lie algebra is not true for solvable Lie superalgebras.

Further for a given nilpotent Lie superalgebra $\mathcal{L}=\mathcal{L}_0\oplus \mathcal{L}_1$, we shall use notations  $\{x_1, x_2, \dots, x_n\}$ and $\{y_1, y_2, \dots, y_m\}$ for the bases of even and odd parts of $\cal L$, respecitvely.

We give an example of such a solvable Lie superalgebra presented in \cite{Wang2002}.

\begin{exam}\label{exam} Let $\mathcal{R}=\mathcal{R}_0\oplus \mathcal{R}_1$ be a four dimensional solvable Lie superalgebra with the following table of multiplications in the basis $\{x_1, x_2, y_1, y_2\}$
$$\left\{\begin{array}{llll}
[x_1,x_2]=x_2,&[x_1,y_2]=y_2,&[y_1,y_1]=-2x_1,\\[1mm]
 [x_2,y_1]=y_2,&[y_1,y_2]=x_2,
\end{array}\right.$$

Clearly, $\mathcal{R}^2=\Span\{x_1, x_2, y_2\}$, while the nilradical of $\mathcal{R}$ is spanned by $\{x_2, y_2\}$. So, $\mathcal{R}^2$ does not belong to the nilradical of $\mathcal{R}$. Moreover, because of
$\Span\{x_1, x_2\}=[\mathcal{R}_1, \mathcal{R}_1]\nsubseteq [\mathcal{R}_0,\mathcal{R}_0]=\Span\{x_2\},$ we conclude that an analogue of Lie's theorem does not hold for the superalgebra $\mathcal{R}.$
\end{exam}
Further we shall use the following result on Lie algebras \cite{Mostow}.
\begin{thm} \label{thmMostow}  Any two maximal fully reducible subalgebras of a linear Lie algebra $L$ are conjugate under an inner automorphism from the radical of $[L,L]$.
\end{thm}

\section{Some properties of superderivations}

In this section we establish some properties of superderivations.
First we analyze the next two examples.

\begin{exam} The straightforward computations lead that any even superderivation $d_0$ of the superalgebra $\mathcal{R}$ from Example \ref{exam} acts $\mathcal{R}$ as follows:
$$d_0(x_1)=\lambda x_2, \quad d_0(x_2)=\mu x_2, \quad d_0(y_1)=\lambda y_2, \quad d_0(y_2)=\mu y_2,$$
while any odd superderivation $d_1$ acts by
$$d_1(x_1)=\gamma y_2 \quad d_1(x_2)=\delta y_2, \quad d_1(y_1)=\eta x_1 - \gamma x_2, \quad
d_1(y_2)=(\delta - \eta)x_2.$$

It is clear that nilradical of $\mathcal{R}$ is $\Span\{x_2,y_2\}$ and $d_0(\mathcal{R}_0)\subseteq \mathcal{N}$, while $d_1(\mathcal{R}_1)\nsubseteq \mathcal{N}.$ So, in the case of $\mathcal{R}^2\nsubseteq N$ an analogue of the classical result of the Lie algebras theory that any derivation of finite-dimensional Lie algebra maps its radical to nilradical, is not true.
\end{exam}

The answer to the question whether even superderivation maps radical to nilradical is follows from the next example.

\begin{exam} \cite{Repovs}\label{exam2222} Let $\mathcal{R}$ be a solvable Lie superalgebra with the following multiplications table
$$\left\{\begin{array}{llll}
[x_1,x_3]=2x_3,&[x_2,y_2]=-2y_2,& [x_2,y_3]=2y_3,&[y_2,y_3]=-x_3,\\[1mm]
[x_1,y_3]=2y_3,&[x_2,y_1]=2y_1, &[x_3,y_1]=2y_3,&[y_1,y_2]=x_1,\\[1mm]

\end{array}\right.$$
where $\mathcal{R}_0=\Span\{x_1,x_2,x_3\}$ and $\mathcal{R}_1=\Span\{ y_1,y_2,y_3\}$. One can check that  $\mathcal{N}=\Span\{x_3,y_3\}$ is the nilradical of $\mathcal{R}$ and $\mathcal{R}^2\nsubseteq \mathcal{N}.$ Note that $ad_{x_2}\in Der(\mathcal{R})_0$ and $ad_{x_2}(y_2)=2y_2\notin \mathcal{N}.$ Therefore, the assertion: any even superderivation of a
Lie superalgebra maps radical to nilradical, is not true.
\end{exam}
Clearly, for an even superderivation $d$ of a Lie superalgebra $\mathcal{L}$ and any $r\in \mathbb{N}$ the following equality holds true:
\begin{equation}\label{eqder}
d^{r}([x, y])=\sum_{i=0}^{r}\begin{pmatrix} r\\i\end{pmatrix}[d^{i}(x),d^{r-i}(y)], \quad x,y\in \mathcal{L}
\end{equation}
where $\begin{pmatrix} r\\i\end{pmatrix}=\frac{r!}{i!(r-i)!}$ is a binomial coefficient.

In case of $\tilde{d}$ is odd superderivation we have  $[\tilde{d},\tilde{d}]=2\tilde{d}^2$, which imply that $\tilde{d}^2$ is even superderivation. Therefore, from Equality \eqref{eqder} we get
\begin{equation}\label{eqder1}
\tilde{d}^{2r}([x, y])=\sum_{i=0}^{r}\begin{pmatrix} r\\i\end{pmatrix}[\tilde{d}^{2i}(x),\tilde{d}^{2(r-i)}(y)], \quad x, y\in \mathcal{L}.\\[1mm]
\end{equation}

Applying Equality \eqref{eqder1}, by induction one can prove the following equality:
$$\tilde{d}^{2r+1}([x, y])=\sum_{i=0}^{r}\begin{pmatrix} r\\i\end{pmatrix}([\tilde{d}^{2(r-i)+1}(x),\tilde{d}^{2i}(y)]+
(-1)^{|x|}[\tilde{d}^{2(r-i)}(x),\tilde{d}^{2i+1}(y)]), \quad x,y\in \mathcal{L}.$$

Applying Equality \eqref{eqder} similarly as in the Lie algebras case (see Theorem 2.5.13 in \cite{Winter}) one can prove the following result.

\begin{prop} \label{prop43} Let $\mathcal{L}$ be a Lie superalgebra with radical $\mathcal{R}$. Then
$d(\mathcal{R})\subseteq \mathcal{R}$ for any even superderivation $d$ of $\mathcal{L}$.
\end{prop}

Since an even superderivation of a Lie superalgebra $\mathcal{L}$ is just a derivation of $\mathcal{L}$, the proofs of the following results are similar to the case of Lie algebras \cite{Gantmacher}, \cite{Jac}.

\begin{thm} \label{thm31}
Let $d$ be an even superderivation of a Lie superalgebra $\mathcal{L}$. Then there exists a unique diagonalizable even superderivation $d_0$ and a unique nilpotent even superderivation $d_1$ such that $d=d_0+d_1$ and $d_0 \circ d_1=d_1 \circ d_0$.
\end{thm}

\begin{lem}\label{lem1}
Let $\mathcal{L}$ be a Lie superalgebra with an even superderivation $d$ and let $\mathcal{L}=\mathcal{L}_{\alpha}\oplus \mathcal{L}_{\beta} \oplus \cdots \oplus \mathcal{L}_{\gamma}$ be its weight decomposition with respect to $d$. Then
$$[\mathcal{L}_{\alpha_1},\mathcal{L}_{\alpha_2}]\subset \left\{\begin{array}{lll}
\mathcal{L}_{\alpha_1+\alpha_2}, &\text{if}\ \alpha_1+\alpha_2\  \text{is an eigenvalue of }\ d,\\[1mm]
0, &  \text{if}\ \alpha_1+\alpha_2\  \text{is not an eigenvalue of }\ d.
\end{array}\right.$$
where $\alpha_1,\alpha_2 \in \{\alpha, \beta, \dots \gamma\}.$
\end{lem}

Now we give the definition of torus of nilpotent Lie superalgebra, which plays an important role in the structure of solvable Lie superalgebras (see \cite{Wang2002}, \cite{Wang2003}).

\begin{defn} A torus on a Lie superalgebra $\mathcal{L}$ is a commutative subalgebra of $Der(\mathcal{L})$ consisting of semisimple endomorphisms. A torus is said to be maximal if it is not strictly contained  in any other torus. We denote by $\mathcal{T}_{max}$ a maximal torus of a Lie superalgebra $\mathcal{L}$.
\end{defn}

Note that in the case of the complex numbers field a semisimple endomorphism is diagonalizable. Moreover, since diagonalizability of odd superderivation leads to its nullity one can assume that in the case of  the complex numbers field a torus on a Lie superalgebra $\mathcal{L}$ is a commutative subalgebra of $Der(\mathcal{L})_0$ consisting of diagonalizable endomorphisms. In addition, choosing an appropriate basis of a Lie superalgebra we can bring the mutually commuting diagonalizable derivations into a diagonal form, simultaneously.

\begin{exam} \label{exampletoralsolvable} Let $\mathcal{N}$ be a nilpotent Lie superalgebra and $\mathcal{T}$ be its a torus.    Consider a non-nilpotent solvable Lie algebra $R_{\mathcal{T}}=\mathcal{N}\dot{+} \mathcal{T}$ with the products $[\mathcal{N}, \mathcal{T}]$ and $[\mathcal{T}, \mathcal{T}]$ to be defined as follows
$$[x, d]=d(x), \quad [d, d']=0 \quad \mbox{for} \quad x\in \mathcal{N} \quad \mbox{and} \quad d, d'\in \mathcal{T}.$$
\end{exam}

Recall that a nilpotent Lie superalgebra $\mathcal{N}$  satisfying the condition $\dim\mathcal{T}_{max}=\dim(\mathcal{N}/\mathcal{N}^2)$ (that is, the dimension of its maximal torus equals the number of generator basis elements of $\mathcal{N}$) is called {\it of maximal rank} (see \cite{Wang2003}).

Here are examples of nilpotent Lie superalgebras of maximal rank.

\begin{exam}\label{exam111}
\

\begin{itemize}
\item[1.] Consider three dimensional Lie superalgebra $\cal N^{1}$ with only non-zero product $[y_1,y_2]=x_1$. Then $\{y_1,y_2\}$ are generators of $\cal N^1$ and $\Span\{ \diag(1, 1, 0), \diag(1, 0, 1)\}$ forms its maximal a torus.
 \item[2.] Let $\cal N^2$ be a five dimensional Lie superalgebra with its table of multiplications:
 $$\left\{\begin{array}{ll}
 [y_1,x_1]=y_{2},\\[1mm]
 [y_2,x_1]=y_{3},\\[1mm]
 [y_{3},y_1]=x_2,\\[1mm]
 [y_{2},y_2]=-x_2.\end{array}\right.$$
 Then $\{x_1,y_1\}$ are the generators of $\cal N^2$ and $\mathcal{T}_{max}=\Span\{\diag(1, 2, 0, 1, 2), \diag(0, 2, 1, 1, 1)\}$.

 \item[3.] Consider an $(m+2)$-dimensional ($m$ is odd) Lie superalgebra $\cal N^3$  with the table of multiplications as follows
$$\left\{\begin{array}{lll}
[y_i,x_1]=y_{i+1},& 1\le i\le m-1,\\[1mm]
[y_{m+1-i},y_i]=(-1)^{i+1}x_2,& 1\le i\le \frac{m+1}{2}.
\end{array}\right.$$
Then $\{x_1,y_1\}$ is the set of generators of $\cal N^3$ and $$\mathcal{T}_{max}=\Span\{\diag(1, 0, 2, \dots, i-2, \dots, n-2,\frac{1}{2}(n-2)), \diag(0, 1, 1, \dots, 1,\frac{1}{2})\}.$$
\end{itemize}
\end{exam}

For completeness we present also an example of a nilpotent Lie superalgebra of not maximal rank.

\begin{exam}\label{exam222} The $(n+1)$-dimensional nilpotent Lie superalgebra $\cal N^4$ defined by the products:
$$[x_1,x_i]=x_{i+1},\ 2\le i\le n-1, \ [y_1,y_1]=x_n,$$
has only three generators $\{x_1,x_2,y_1\},$ while
$\mathcal{T}_{max}=\Span\{\diag(1, 0, 2, \dots, i-2, \dots, n-2,\frac{1}{2}(n-2)), \diag(0, 1, 1, \dots, 1,\frac{1}{2})\}$ is two dimensional.
\end{exam}

In view of the fact that a fully reducible complex Lie algebra is the direct sum of a semisimple Lie algebra and an abelian algebra of diagonalizable elements \cite{Jac1}, we conclude that an abelian algebra of diagonalizable elements is a fully reducible. Therefore, one can conclude that $\cal T_{max}$ is a maximal fully reducible subalgebra of the linear Lie algebra $Der(\mathcal{N})_0$. Now applying Theorem \ref{thmMostow} to $\cal T_{max}$ we get the following result.

\begin{thm}
Any two maximal tori of a complex nilpotent Lie superalgebra $\mathcal{N}$ are conjugate under  an inner automorphism from the radical of $[Der(\mathcal{N})_0, Der(\mathcal{N})_0]$
\end{thm}

\begin{cor} \label{corequaldim}
 The dimensions of any two maximal tori of a nilpotent Lie superalgebra
 are equal.
\end{cor}

Note that a basis of $\mathcal{N}$ can be chosen a such way that any non-generator basis element is represented as right-normed product of the generator basis elements. Such a basis of $\mathcal{N}$ is called {\it the natural basis}. For convenience, further we shall consider only the natural basis of $\mathcal{N}$.

From now on we shall consider a complex solvable (non-nilpotent) Lie superalgebra such that its square is embedded into its nilradical. Let us consider a such solvable Lie superalgebra $\mathcal{R}$ as a direct sum of vector subspaces $\mathcal{N}$ and $\mathcal{Q}$, where $\mathcal{N}$ is its nilradical and $\mathcal{Q}$ is a complementary subspace, i.e., $\mathcal{R}=\mathcal{N}\oplus \mathcal{Q}$.

Let $\mathcal{Q}=\mathcal{Q}_0\oplus \mathcal{Q}_1$ be a decomposition  such that $\mathcal{Q}_i\subseteq \mathcal{R}_i$ for $i=0,1$.

\begin{prop}\label{prop3.12} Let $\mathcal{R}=\mathcal{N}\oplus \mathcal{Q}$ be a solvable Lie superalgebra with nilradical $\mathcal{N}=\mathcal{N}_0\oplus \mathcal{N}_1$ and complementary subspace $\mathcal{Q}=\mathcal{Q}_0\oplus \mathcal{Q}_1$. Then $\mathcal{Q}_1=\{0\}$.
\end{prop}
\begin{proof} Assume the contrary, that is there exists nonzero element $0\neq x\in \mathcal{Q}_1$. Set $\mathcal{V}=\mathcal{N}\oplus \mathbb{C}x$. Since $\mathcal{N}$ is an ideal and $\mathcal{R}^2\subseteq \mathcal{N}$, we derive that $\mathcal{V}$ is also an ideal. Moreover, $\mathcal{V}=\mathcal{V}_0\oplus \mathcal{V}_1$, where $\mathcal{V}_0=\mathcal{N}_0$ and $\mathcal{V}_1=\mathcal{N}_1\oplus \mathbb{C}x.$
Due to Theorem \ref{Gilg} the nilpotency of $\mathcal{N}$ implies the existence of $p, q\in \mathbb{N}$ such that ${\cal C}^p(\mathcal{N}_0)={\cal C}^q(\mathcal{N}_1)=\{0\}.$ Taking into account
${\cal C}^p(\mathcal{V}_0)={\cal C}^p(\mathcal{N}_0)=\{0\}$ and
$${\cal C}^{q+1}(\mathcal{V}_1)={\cal C}^{q+1}(\mathcal{N}_1)+{\cal C}^{q+1}(x)={\cal C}^{q}([x,\mathcal{N}_0])\subset{\cal C}^q({\mathcal{N}_1})=\{0\}$$
we obtain the nilpotency of the ideal $\mathcal{V}$, which contradicts to the maximality of $\mathcal{N}$. Therefore, $\mathcal{Q}_1=\{0\}.$
\end{proof}

\begin{prop}\label{prop3.13} Let $\mathcal{R}=\mathcal{N}\oplus \mathcal{Q}_0$ be a solvable Lie superalgebra with nilradical $\mathcal{N}=\mathcal{N}_0\oplus \mathcal{N}_1$. Then $ad_{x|\mathcal{N}}$ is non-nilpotent operator for any $x\in \mathcal{Q}_0.$
\end{prop}

\begin{proof} Due to a result of \cite{Kac} (see Proposition 1.3.3) it is known that a Lie superalgebra is solvable if and only if its even part is solvable. Therefore, the even part of $\mathcal{R}$ is a solvable Lie algebra and thanks to \cite{Mub} we obtain that for an arbitrary $0\neq x\in\mathcal{Q}_0$ the operator $ad_{x|\mathcal{N}_0}$ is non-nilpotent. Since $\mathcal{N}_0$ and $\mathcal{N}_1$ are invariant spaces under the operator $ad_{x}$, the non-nilpotence of $ad_{x|\mathcal{N}_0}$ implies the non-nilpotence of $ad_{x}$. \end{proof}

In order to estimate the codimension of the nilradical of a solvable Lie superalgebra we need the following definition.

\begin{defn} The set $\{d_1, d_2, \dots, d_t\}$ of even superderivations of a Lie superalgebra $\mathcal{L}$ is said to be nil-independent, if the nilpotence of the superderivation $\alpha_1 d_1 + \alpha_2 d_2 + \dots + \alpha_n d_n$ implies $\alpha_i=0, 1 \leq i \leq t$.
\end{defn}

Due to Propositions \ref{prop3.12} - \ref{prop3.13} we conclude that the dimension of $\mathcal{Q}_0$ is not greater than the number of nil-independent even superderivations of $\mathcal{N}$. Therefore, in the same way as in \cite{Snobl} one can prove the following result.

\begin{thm}\label{thm15} Let $\mathcal{R}$ be a solvable Lie superalgebra with nilradical $\cal N$ such that $\mathcal{R}^2\subseteq \mathcal{N}.$ Then
$$dim (\cal R/\cal N) \leq\dim(\cal N/\cal N^2).$$
\end{thm}

\begin{rem} \label{rem3.16} It should be noted that the condition $\mathcal{R}^2\subseteq \mathcal{N}$ in Theorem \ref{thm15} is essential. Indeed, if we consider the Lie superalgebra $\cal R$ from Example \ref{exam2222}, then $\cal R^{2}\nsubseteq \cal N$ and $dim (\cal R/\cal N)=4 > 2=\dim(\cal N/\cal N^2).$
\end{rem}

Let us assume that $\left\{x_1,\dots,x_k,y_1,\dots,y_s\right\}$ are generators basis elements of $\mathcal{N}.$ Set $X=\Span\left\{x_1,\dots,x_k\right\}$ and $Y=\Span\left\{y_1,\dots,y_s\right\}$.

Let $D$ be an even superderivation of $\mathcal{N}.$ Then its matrix form in a basis of $\cal N$ which agreed with natural filtration \eqref{filtr} has the following form:
\begin{equation}\label{eq222}
\left(\begin{array}{ccccccccccccc}
  D^0_{11} & D^0_{12} & \dots & D^0_{1q}& \vline &&&&\\
         & D^0_{22}   & \dots & D^0_{2q}& \vline &&O&&\\
         && \ddots  & \vdots& \vline &&&&\\
         &         &        & D^0_{qq}& \vline &&&&\\
         \hline
         &&&& \vline&  D^1_{11} & D^1_{12} & \dots & D^1_{1p}\\
         &&&& \vline&         & D^1_{22}   & \dots & D^1_{2p}\\
         &&O&& \vline&          &&\ddots  & \vdots\\
         &&&& \vline& &         &        & D^1_{pp}
  \end{array}\right),
\end{equation}
where submatrices $D^0_{11}\in M_{k,k}$ and $D^1_{11}\in M_{s,s}$ are restrictions and projections of $D$ on $X$ and $Y$, respectively.

\begin{rem} \label{rem111} Note that $D$ is nilpotent if and only if its submatrices  $D_{11}^{0}$ and $D_{11}^{1}$ are nilpotent. One of these implications is obvious; the other is derived as follows. Applying equation (\ref{eqder}) with assuming the existence of $\lambda, \mu \in \mathbb{N}$ such that
$\big({D_{11}^{0}}\big)^{\lambda}=\big({D_{11}^{1}}\big)^{\mu}=0$ we deduce the existence of $\tau\in \mathbb{N}$ such that
$\big(D^{0}_{ll}\big)^{\tau}=\big(D^{1}_{tt}\big)^{\tau}=0$ for all  $l=2,\dots, q$ and $t=2,\dots, p.$ This implies that
the block upper triangular matrix of $D^{\tau}$ has vanishing diagonal blocks and is consequently nilpotent, implying also the nilpotency of $D$ itself.
\end{rem}

Due to Proposition \ref{prop3.12} we have $\cal R=\cal N\oplus \cal Q_0$. We decompose $Q_0$ as follows: $\cal Q_0=\cal Q_0^1\oplus \cal Q_0^2,$
where $\cal Q_0^2=\{z\in \cal Q_0 \ | \ {ad_z}_{|X} - \ \mbox{is nilpotent}\}$ and $\cal Q_0^1$ is complementary subspace to $\cal Q_0^2$.
Thanks to Proposition \ref{prop3.13} and Remark \ref{rem111} we derive that ${ad_z}_{|Y}$ is not nilpotent for any $z\in \cal Q_0^2$.

Next result presents upper bounds on dimensions of $\cal Q_{0}^{1}$ and $\cal Q_0^{2}.$

\begin{thm}\label{thm3.17} For a solvable Lie superalgbera $\cal R=\cal N_{0}\oplus \cal N_{1}\oplus\cal Q_0^1\oplus \cal Q_0^2$ with nilradical $\mathcal{N}$ such that $\cal R^2\subseteq \cal N$ the following inequalities hold true
$$\dim\cal Q_{0}^{1}\leq\cal \dim \Big(\cal N_{0}/\big(\cal N_{0}^{2}+[\cal N_{1}, \cal N_{1}]\big)\Big) \quad \mbox{and} \quad \dim\cal Q_{0}^{2}\leq\cal \dim (\cal N_{1}/[\cal N_{1},\cal N_{0}]).$$
\end{thm}
\begin{proof}

Let $D$ be an even superderivation of $\mathcal{N}.$
Then, we choose a basis $\Gamma=\left\{x_1,\dots x_{n},y_1,\dots ,y_{m}\right\}$ of the nilpotent algebra $\cal N$ that is agreed with natural filtration (\ref{filtr}). The main advantage of working in the basis $\Gamma$ lies in the fact that any  even superderivation $D,$ is fully specified once its action on the generator basis elements of $\cal N.$ Moreover, the matrix of any even superderivation of $\cal N$ is upper block triangular in the basis $\Gamma$ has the form \eqref{eq222}.

Due to the relation $D([a,b])=[D(a),b]+[a,D(b)]$ we conclude that elements of the diagonal blocks $D^{0}_{ll},$ $(l=2,\dots, q)$ and $D^{1}_{tt},$ $(t=2,\dots, p)$ are linear functions of elements of $D_{11}^0$, $D_{11}^1$ and structure constants of $\mathcal{N}$.

The condition $\mathcal{R}^2\subseteq\mathcal{N}$ leads that $[ad_{z_1},ad_{z_2}]=ad_{[z_1,z_2]}\in Inder(\mathcal{N})$
for any $z_{1},z_{2}\in\mathcal{Q}_{0}.$

Consider a set of even superderivations $\{D_{1},\dots,D_{l}\}$ of $\mathcal{N}$ such that
$[D_i,D_j]\in Inder(\mathcal{N})$ for any $i,j\in \{1, \dots, l\}.$ This, implies
\begin{align}\label{commm}
[D_{ll}^{0i},D_{ll}^{0j}]=0, && [D_{tt}^{1i},D_{tt}^{1j}]=0,
\end{align}
where $l=1,\dots,q, t=1,\dots,p $ and $i,j\in \{1,\dots, l\}.$

The equivalence from Remark \ref{rem111} leads that even superderivations $D_1,\dots, D_l$ are linearly nil-independent if and only if their cooresponding submatrices
$\{D_{11}^{01},\dots,D_{11}^{0l}\}$ and $\{D_{11}^{11},\dots,D_{11}^{1l}\}$ are linearly nil-independent. Together with Equalities \eqref{commm} it means that the number of linearly nil-independent outer even superderivations of $\cal N$ acting nilpotently on
$\left\{y_1,\dots,y_s\}\right.$ and commuting to inner even superderivations is bounded from above by the maximal number of linearly nil-independent commuting matrices of dimension $k\times k$. This number is equal to $\dim \cal N_0 - \dim \Big([\cal N_0,\cal N_0] + [\cal N_1,\cal N_1]\Big),$ finishing the first estimation.

Similarly, the number of linearly nil-independent outer even superderivations of $\cal N$ acting nilpotently on $\left\{x_1,\dots, x_k\}\right.$ and commuting to inner even superderivations is bounded from above by the maximal number of linearly nil-independent commuting matrices of dimension $s\times s$. This number is equal to $\dim \cal N_1 - \dim ([\cal N_1,\cal N_0]),$ which  finished the second estimation.
\end{proof}

Let us introduce the notion of a {\it characteristically nilpotent Lie superalgebra} as superalgebra whose all superderivations are nilpotent. This notion is agreed with the notion of characterictically nilpotent Lie algebra. For detailed results on characteristically nilpotent Lie algebras we refer reader to \cite{CharNil} and references therein. It should be noted that Theorem \ref{thmEngel} implies that any characteristically nilpotent Lie superalgebra is nilpotent.

Here is an example of characteristically nilpotent Lie superalgebra.
\begin{exam} \label{exam33} Consider an eight dimensional nilpotent Lie superalgebra
$\mathcal{N}=\mathcal{N}_0\oplus \mathcal{N}_1$ with the following table of multiplications:
$$\left\{\begin{array}{llll}
[x_i,x_1]=x_{i+1},\ 2\le i\le 6,& [x_2,x_5]=x_7,& [x_2,x_3]=x_6, \\[1mm] [x_3,x_4]=-x_7, & [x_2,x_4]=x_7, &[y_1,y_1]=x_7.\\[1mm]
\end{array}\right.$$
A straightforward computations lead that the matrix form of superderivations of $\mathcal{N}$ has the form
$$\left(\begin{array}{cccccccc}
0 & 2b_4 & a_3 & a_4 & a_5 & a_6 & a_7&\alpha_1 \\
0 & 0 & b_3 & b_4 & b_5 & b_6 & b_7&0 \\
0 & 0 & 0&b_3 & b_4 & b_5-a_3 & b_6-a_4-a_5 &0\\
0 & 0 & 0&0&b_3 &3b_4 & b_5-a_3+a_4 &0\\
0 & 0 & 0&0&0&b_3 &5b_4-a_3&0 \\
0 & 0 & 0&0&0&0&b_3 +2b_4 &0\\
0 & 0 & 0&0&0&0&0 &0\\
0 & 0 & 0&0&0&\alpha_1&\beta_7 &0\\
\end{array}\right).$$
It is easy to see that for any $d\in Der(\mathcal{N})$ one has $d^3=0$, i.e., $\mathcal{N}$ is a characteristically nilpotent Lie superalgebra.
\end{exam}

The characteristically nilpotency of a Lie superalgebra implies the characteristically nilpotency of its even part, while the converse is not true, in general.

Indeed, if we consider an nine dimensional Lie superalgebra with the table of multiplication as follows
$$\left\{\begin{array}{llll}
[x_1,x_i]=x_{i+1},\ 2\le i\le 6, & [x_2,x_3]=x_6, & [x_2,x_4]=x_7,\\[1mm]
[x_2,x_5]=x_7,& [x_3,x_4]=-x_7, & [y_1,y_2]=x_7,\\[1mm]
\end{array}\right.$$
then even superderivations of the superalgebra have the following matrix form
$$\left(\begin{array}{ccccccccc}
0 & 2b_4 & a_3 & a_4 & a_5 & a_6 & a_7& 0&0 \\
0 & 0 & b_3 & b_4 & b_5 & b_6 & b_7&0&0 \\
0 & 0 & 0&b_3 & b_4 & b_5-a_3 & b_6-a_4-a_5&0&0 \\
0 & 0 & 0&0&b_3 &3b_4 & b_5-a_3+a_4 &0&0\\
0 & 0 & 0&0&0&b_3 &5b_4-a_3 &0&0\\
0 & 0 & 0&0&0&0&b_3 +2b_4&0&0 \\
0 & 0 & 0&0&0&0&0&0&0 \\
0 & 0 & 0&0&0&0&0&c_1&0 \\
0 & 0 & 0&0&0&0&0&0&-c_1\\
\end{array}\right),$$
which is not nilpotent with $c_1\neq 0.$ Hence, this is an example of non-characteristically nilpotent Lie superalgebra with characteristically nilpotent even part.

\begin{prop}\label{p31} Let $\mathcal{R}=\mathcal{N}\oplus \mathcal{Q}_0$ be a solvable Lie superalgebra. If $\mathcal{N}_0$ is a characteristically nilpotent Lie algebra, then $\mathcal{Q}_0=\mathcal{Q}_0^2$.
\end{prop}
\begin{proof} Let $\mathcal{N}_0$ is a characteristically nilpotent Lie algebra. Consider $z\in \mathcal{Q}_0$, then $ad(z)_{|\mathcal{N}_0}\in Der(\mathcal{N}_0)$. Therefore, for any $z\in \mathcal{Q}_0$ we derive that $ad(z)_{|\mathcal{N}_0}$ is nilpotent map. Hence, $\mathcal{Q}_0=\mathcal{Q}_0^2.$
\end{proof}

\begin{rem} But the condition $\mathcal{Q}_0=\mathcal{Q}_0^2$ does not always mean that $\mathcal{N}_0$ is a characteristically nilpotent Lie algebra. Indeed, if we consider the nilpotent Lie superalgebra $\cal N$ of maximal rank with non-characteristically nilpotent of even part and with the table of multiplications: $[y_1,y_1]=x_1, \quad [y_2,y_2]=x_2,$ then one can check that $\cal Q_0=\cal Q_0^2.$
\end{rem}

\section{Description of complex solvable Lie superalgebras of maximal rank}

In this section we study the structure of maximal dimensional solvable Lie superalgebra $\cal R$ with nilradical $\cal N$ of maximal rank and under the conditions: $\cal R^2 \subseteq \cal N$ and $[\cal R_1, \cal R_1] \subseteq [\cal R_0, \cal R_0]$.

For an nilpotent Lie superalgebra of maximal rank one can associate a solvable Lie superalgebra of maximal rank \cite{Wang2003}.

\begin{defn} A solvable Lie superalgebra $\mathcal{R}_{\mathcal{T}_{max}}$ is called {\it of maximal rank} if $\dim \mathcal{T}_{max}=\dim(\mathcal{N}/\mathcal{N}^2)$.
\end{defn}

This notion is agreed with the notion of solvable Lie algebra of maximal rank \cite{Meng}. In fact, in \cite{Khal} it was shown that a complex solvable Lie algebra of maximal rank can also be defined as a solvable Lie algebra $\mathcal{R}$ with the condition $\codim \mathcal{N}=\dim(\mathcal{N}/\mathcal{N}^2).$

 Namely, the following result was obtained.
\begin{thm} \label{maxrankalgebra}  A complex finite-dimensional solvable Lie algebra with nilradical $\mathcal{N}$ of maximal rank and complementary subspace $\mathcal{Q}$  to the nilradical $\mathcal{N}$ of dimension equal to the rank of $\mathcal{N}$ is isomorphic to an algebra $\cal R_{\mathcal{T}_{max}}$.
\end{thm}

Thanks to Theorem \ref{thm3.17} we conclude that a solvable Lie superalgebra $\cal R=\cal N_{0}\oplus \cal N_{1}\oplus\cal Q_0$ with $\cal Q_0=\cal Q_0^1\oplus \cal Q_0^2$ and nilradical $\cal N$ of maximal rank has the maximal dimension only in case of
$$\dim\cal Q_{0}^{1}=\cal \dim \Big(\cal N_{0}/\big(\cal N_{0}^{2}+[\cal N_{1}, \cal N_{1}]\big)\Big) \quad \mbox{and} \quad \dim\cal Q_{0}^{2}=\cal \dim (\cal N_{1}/[\cal N_{1},\cal N_{0}]).$$

Therefore, the maximal dimension of a solvable Lie superalgebra with nilradical of maximal rank is equal to $\dim \cal N+\dim \cal T_{max},$ i.e., $\dim Q_0=\dim \cal T_{max}$.

\begin{rem} \label{rem222} Straightforward computations show that solvable Lie superalgebras with nilradicals $\cal N^1$ and $\cal N^2$ and with dimensions of complementary subspaces to nilradicals equal to the numbers of generators basis elements of nilradicals are isomorphic to superalgebras of the form $\cal R_{\cal T_{max}}$. Moreover, from \cite{Khu} it is known that a solvable Lie superalgebra with nilradical $\cal N^3$ such that $\codim \cal N^3= \dim \cal T_{max}$ is isomorphic to a superalgebra $\cal R_{\cal T_{max}}$, where $\cal T_{max}$ is a maximal torus of $\cal N^3$.
\end{rem}

Let $\cal N$ be a nilpotent Lie superalgebra of maximal rank, that is,
$$\dim\cal (N / \cal N^2)=\dim \Big(\cal N_{0}/\big(\cal N_{0}^{2}+[\cal N_{1}, \cal N_{1}]\big)\Big)+\dim (\cal N_{1}/[\cal N_{1},\cal N_{0}])=\dim \cal T_{max}.$$

\begin{defn}
A solvable Lie superalgebra $\mathcal{R}$ with nilradical $\cal N$ is called a  maximal solvable extension of the nilpotent Lie superalgebra $\mathcal{N}$, if $\codim \cal N$ is maximal.
\end{defn}

Since we consider the case of solvable Lie superalgebras for which Lie's Theorem is applicable, then there exists a basis of $\mathcal{R}$ such that all operators ${ad_{z}}, \ z\in \mathcal{R}$ have upper-triangular matrix form. In particular, all operators $ad_{z|\mathcal{N}}$ $(z\in \mathcal{Q}_0)$ in the basis of $\mathcal{N}$ have upper-triangular matrix form too.

According to Theorem \ref{thm31} an operator $ad_{{z}|\mathcal{N}}, \ z\in \mathcal{Q}_0$ admits a decomposition $ad_{{z}|\mathcal{N}}=d_{0}+d_{1},$ where $d_{0}$ is diagonalizable and $d_{1}$ is nilpotent even superderivations of $\mathcal{N}$ such that $[d_{0}, d_{1}]=0$.

We put
    $$ad_{{z}|\mathcal{N}}=\left(\begin{array}{cccccccc}
    \gamma_{1}&*&*&\dots &*&*\\[1mm]
    0&\gamma_{2}&*&\dots &*&*\\[1mm]
    0&0&\gamma_{3}&\dots &*&*\\[1mm]
    \vdots&\vdots&\vdots&\ddots &\vdots&\vdots\\[1mm]
    0&0&0&\dots &\gamma_{n-1}&*\\[1mm]
    0&0&0&\dots &0&\gamma_{n}\\[1mm]
    \end{array}\right).$$

    Due to the upper-triangularity of $ad_{{z}|\mathcal{N}}$ we conclude that $d_0=\diag(\gamma_{1},  \dots, \gamma_{n}).$

\begin{prop} \label{prop3.20}Let $\cal R=\cal N \oplus \cal Q_0$ be a maximal solvable extension of nilpotent Lie superalgebra of maximal rank. Then $\dim \cal Q_0 = \dim\cal T_{max}.$
\end{prop}
\begin{proof} Clearly, $\dim \cal Q_0 \geq \dim \cal T_{max}$ (otherwise, $\cal R$ is not a maximal extension of $\cal N$).
Let $\{z_1, \dots, z_p\}$ be a basis of $\cal Q_0$. Then diagonals of $ad_{{z_i}|\mathcal{N}}, \ 1\leq i \leq p$ are derivations of $\cal N$. Thanks to Proposition \ref{prop3.13} an operator $ad_{{z}|\mathcal{N}}$ is non nilpotent for any non-zero $z\in \cal Q_0$. Therefore, we get the existence $p$ linear independent diagonal derivations. Therefore, $\dim \cal Q_0\leq \dim \cal T_{max}$.
\end{proof}

\begin{rem}\label{rem3.21} From the above proof we deduce that diagonal elements of $ad_{{z}|\mathcal{N}}$ for the general $z\in \cal Q_0$ is nothing else but roots of action $\cal T_{max}$ on $\cal N$.
\end{rem}

Set
$$\dim \Big(\cal N_{0}/\big(\cal N_{0}^{2}+[\cal N_{1}, \cal N_{1}]\big)\Big)=k, \quad \dim (\cal N_{1}/[\cal N_{1},\cal N_{0}])=s.$$

\

From Theorem \ref{thm3.17} and Proposition \ref{prop3.20} we derive that $\cal Q_0^1$ and $\cal Q_0^2$ admit basis $\{z_1, \dots, z_k\}$ and
$\{z_{k+1}, \dots, z_{k+s}\}$, respectively, such that each $ad_{{z_i}|\mathcal{N}}, \ 1\leq i \leq k+s$ has upper-triangular form. Taking into account that subspaces $\cal N_0$ and $\cal N_1$ are invariant under $ad_{{z_i}|\mathcal{N}}$ we have the following products modulo $\cal N^2$:
\begin{equation}\label{eq1}
\begin{array}{lllll}
[x_i,z_j]& \equiv & \delta_{i,j}x_i+\sum\limits_{p=i+1}^k\mu_{i,j}^px_p,& 1\leq i \leq k, & 1\leq j \leq k+s,\\[1mm]
[y_i,z_j]& \equiv &\delta_{i,j-k}y_i+\sum\limits_{q=i+1}^s\nu_{i,j}^qy_q,& 1\leq i \leq s, &1\leq j \leq k+s,\\[1mm]
\end{array}
\end{equation}
where $\delta_{i,j}$ is Kronecker delta.

\begin{lem} \label{lem3.22} There exists a basis of $\cal R$ such that for any $i,j\in \{1, \dots, s\}$ with $i\neq j$ the following products modulo $\cal N^2$ hold true:
$$\begin{array}{lllllllllll}
i) & [x_i,z_j] \equiv \delta_{i,j}x_i, & 1\leq i \leq k,& 1\leq j \leq k+s,\\[1mm]
ii)& [y_i,z_j] \equiv \delta_{i,j-k}y_i, & 1\leq i \leq s, & 1\leq j \leq k+s,\\[1mm]
iii) &[z_i, z_j] \equiv 0, & 1\leq i, j \leq k+s. &\\[1mm]
\end{array}$$
\end{lem}

\begin{proof} Part {\it i)}. For the case of $i=j$ consider consistently for $i=1, 2, \dots, k$ the following change of generator basis elements that lie in $\cal N_0$:
$x_i'=x_i+\mu_{i,i}^{i+1}x_{i+1}.$ Then we get $[x_i',z_i] \equiv x_i'+(\mu_{i,i}^{i+2})'x_{i+2}+\sum\limits_{t=i+3}^k*x_t.$

Further, putting $x_i''=x_i'+(\mu_{i,i}^{i+2})'x_{i+2},$ we derive $[x_i'', z_i] \equiv x_i''+\sum\limits_{t={i+3}}^k*x_t.$

Continuing in such way one can assume that $[x_i,z_i] \equiv x_i, \ 1\leq i \leq k.$

Let now $i\neq j$. Since $[\cal Q_0, \cal Q_0]\subseteq \cal N_0$ we conclude that $[x_i,[z_i,z_j]] \equiv 0.$ Therefore, applying the product
$[x_i,z_i]\equiv x_i$ and the first congruence of \eqref{eq1} in
$$[x_i,[z_i,z_j]]=[[x_i,z_i],z_j]-[[x_i,z_j],z_i] \equiv [x_i,z_j]-[[x_i,z_j],z_i]$$
we deduce $[x_i, z_j] \equiv 0$ for $1\leq i\leq k$ and $1 \leq j \leq k+s$, $i\neq j$.

The part {\it ii)} can be obtained in a such way as part {\it i)}.

Part {\it iii)}. Taking onto account part {\it i)} we consider the equality
\begin{equation}\label{eq3.5}
[z_t, [z_i,z_j]]=[[z_t, z_i],z_j] - [[z_t, z_j],z_i].
\end{equation}
If indexes $t, i, j$ in \eqref{eq3.5} are pairwise non equal from the set $\{1, \dots, k\}$, then we derive $$[z_i,z_j] \equiv \theta_{i,j}^ix_i+\theta_{i,j}^jx_j.$$
Now, setting $z_i'=z_i-\sum\limits_{p=1, p\neq i}^{k}\theta_{i,p}^px_p,\ 1\leq i\leq k,$ we obtain $[z_i',z_j'] \equiv 0$ with $1\leq i, j \leq k$.

If indexes $i, j$ in \eqref{eq3.5} are pairwise non equal from the set $\{k+1, \dots, k+s\}$ and $t\in\{1, \dots, k\}$, then we deduce $[z_{i},z_{j}] \equiv 0, \ k+1\leq i, j \leq k+s$.

If indexes $t, i$ in \eqref{eq3.5} are pairwise non equal from the set $\{1, \dots, k\}$ and $j\in\{k+1, \dots, k+s\}$, then we get
$[z_{i},z_{j}] \equiv \theta_{i,j}^ix_i.$ Setting $z_{j}'=z_{j}-\sum\limits_{p=1}^{k}\theta_{j,i}^px_p,\ 1\leq i\leq k,$
we obtain $[z_i,z_{j}'] \equiv 0$ with $1\leq i \leq k$ and $k+1\leq j \leq k+s$.
\end{proof}

Let $\tau+1$ is the nilindex of $\cal N$.

\begin{lem}\label{lem3.23} The following congruences modulo $\cal N^t$ with $2\leq t \leq \tau+1$ are true:
$$\begin{array}{llll}
[x_i,z_j]\equiv \delta_{i,j}x_i, & 1\leq i \leq k, & 1\leq j \leq k+s,\\[1mm]
[x_i,z_j]\equiv \alpha_{i,j}x_i, & k+1\leq i \leq n, & 1\leq j \leq k+s,\\[1mm]
[y_i,z_j]\equiv \delta_{i,j}y_i, & 1\leq i \leq s, & 1\leq j \leq k+s,\\[1mm]
[y_i,z_j]\equiv \beta_{i,j}y_i, & s+1\leq i\leq m, & 1\leq j\leq k+s,\\[1mm]
[z_i, z_j]\equiv 0, & 1\leq i, j \leq k+s.\\[1mm]
\end{array}$$
\end{lem}
\begin{proof} We shall prove the statement of lemma by induction on $t$. Due to Lemma \ref{lem3.22} we have the base of induction.
Let assume that the lemma is true for $m$ and we shall prove it for $t+1$. We denote by $\{x_{l}, \dots, x_{l_{t}}\}$ and $\{y_{u}, \dots, y_{u_{t}}\}$ the basis elements of $\cal N$ which lie in $\cal N^{t+1} \setminus \cal N^{t+2}$. Then we have products modulo $\cal N^{t+1}$:
$$\begin{array}{llll}
[x_i,z_j]\equiv \delta_{i,j}x_i+\sum\limits_{p=l}^{l_{m}}\lambda_{i,j}^px_p, & 1\leq i \leq k, & 1\leq j \leq k+s,\\[1mm]
[x_i,z_j]\equiv \alpha_{i,j}x_i+\sum\limits_{p=l}^{l_{m}}\lambda_{i,j}^px_p, & k+1\leq i \leq n, & 1\leq j \leq k+s,\\[1mm]
[y_i,z_j]\equiv \delta_{i,j}y_i+\sum\limits_{p=u}^{u_{m}}\mu_{i,j}^py_p, & 1\leq i \leq s, & 1\leq j \leq k+s,\\[1mm]
[y_i,z_j]\equiv \beta_{i,j}y_i+\sum\limits_{p=u}^{u_{m}}\mu_{i,j}^py_p, & s+1\leq i\leq m, & 1\leq j\leq k+s,\\[1mm]
[z_i, z_j]\equiv \sum\limits_{p=l}^{l_{m}}\theta_{i,j}^px_p, & 1\leq i, j \leq k+s.\\[1mm]
\end{array}$$

From congruences with $1\leq i \leq k$ and $1\leq j, t \leq k+s$:
$$0\equiv[x_i,[z_j,z_t]]=[[x_i,z_j],z_t]-[[x_i,z_t],z_j]\equiv
[\delta_{i,j}x_i+\sum\limits_{p=l}^{l_{m}}\lambda_{i,j}^px_p,z_t]-
[\delta_{i,t}x_i+\sum\limits_{p=l}^{l_{m}}\lambda_{i,t}^px_p,z_j]\equiv$$
$$\delta_{i,j}[x_i,z_t]+\sum\limits_{p=l}^{l_{m}}\lambda_{i,j}^p[x_p,z_t]
-\delta_{i,t}[x_i,z_j]-\sum\limits_{p=l}^{l_{m}}\lambda_{i,t}^p[x_p,z_j]\equiv$$
$$\delta_{i,j}\Big(\delta_{i,t}x_{i}+\lambda_{i,t}^{l}x_{l}\Big)
+\lambda_{i,j}^{l}\alpha_{l,t}x_l-\delta_{i,t}\Big(\delta_{i,j}x_i+\lambda_{i,j}^lx_l\Big)
-\lambda_{i,t}^l\alpha_{l,j}x_l+\sum\limits_{p=l+1}^{l_{m}}(*)x_p\equiv$$
$$\Big(\delta_{i,j}\lambda_{i,t}^l+\lambda_{i,j}^l\alpha_{l,t}-\delta_{i,t}\lambda_{i,j}^l
-\lambda_{i,t}^l\alpha_{l,j}\Big) x_l+\sum\limits_{p=l+1}^{l_{m}}(*)x_p,$$
we deduce $\lambda_{i,j}^l(\alpha_{l,t}-\delta_{i,t})=\lambda_{i,t}^l(\alpha_{l,j}-\delta_{i,j}).$

Taking into account Remark \ref{rem3.21} and the result that roots subspaces in case of primitive roots are one-dimensional (see Lemma 2.5 in \cite{Wang2003}), we conclude that for any $l\neq i$ there exists $t$ such that $\alpha_{l,t}\neq\delta_{i,t}.$

If $\alpha_{l,j}-\delta_{i,j}=0,$ then $\lambda_{i,j}^l=0$.

If $\alpha_{l,j}-\delta_{i,j}\neq0,$ then we have $\lambda_{i,j}^l=\frac{\lambda_{i,t}^l(\alpha_{l,j}-\delta_{i,j})}{\alpha_{l,t}-\delta_{i,t}}.$

Now taking $x_i'=x_i - \frac{\lambda_{i,j}^l}{\alpha_{l,j}-\delta_{i,j}}x_l,$ we derive

$$[x_i', z_j]\equiv
\delta_{i,j}x_i+\lambda_{i,j}^lx_l-
\frac{\lambda_{i,j}^l}{\alpha_{l,j}-\delta_{i,j}}
\alpha_{l,j}x_l+\sum\limits_{p=l+1}^{l_{m}}(*)x_p\equiv\delta_{i,j}x_i'+\sum\limits_{p=l+1}^{l_{m}}(*)x_p.$$

Applying sequentially the same arguments as in above ($l_m-l$)-times we derive
$$[x_i, z_j]=\delta_{i,j}x_i, \quad 1\leq i \leq k, \quad 1\leq j \leq k+s \quad (mod \ \cal N^{m+1}).$$

Similarly, we obtain
$$[y_i, z_j]=\delta_{i,j-k}y_i, \quad 1\leq i \leq s, \quad 1\leq j \leq k+s \quad (mod \ \cal N^{m+1}).$$

Applying in Lie superidentity the fact that a basis is natural, one can conclude that
$$[x_i,z_j]\equiv \alpha_{i,j}x_i, \quad k+1\leq i \leq n, \quad [y_t,z_j]\equiv \beta_{t,j}y_t, \quad  s+1\leq t\leq m, \quad 1\leq j\leq k+s,$$
where
\begin{itemize}
  \item $\alpha_{i,j}$ is the number of entries of a generator basis element $x_j$ (respectively, $y_{j-k}$) involved in forming non generator basis element $x_i$ for $1\leq j \leq k$ (respectively, $k+1\leq j \leq k+s$);
  \item $\beta_{i,j}$ is the number of entries of a generator basis element $x_j$ (respectively, $y_{j-k}$) involved in forming non generator basis element $y_i$ for $1\leq j \leq k$ (respectively, $k+1\leq j \leq k+s$).
\end{itemize}

Now taking the following change of basis elements
$$z_1'=z_1-\sum\limits_{p=l, \alpha_{p,2}\neq 0}^{l_m}\frac{\delta_{\alpha_{p,1}, 0}\theta_{1,2}^p}{\alpha_{p,2}}x_p, \quad
z_i'=z_i + \sum\limits_{p=l, \alpha_{p,1}\neq 0}^{l_m}\frac{\theta_{1,i}^p}{\alpha_{p,1}}x_p, \quad 2\leq i \leq k+s,
$$
one can assume

\begin{equation}\label{eq361}
[z_1,z_2]\equiv\sum\limits_{p=l}^{l_{m}}\delta_{\alpha_{p,1}, 0}\delta_{\alpha_{p,2},0}\theta_{1,2}^px_p, \quad
[z_1,z_i]\equiv \sum\limits_{p=l}^{l_{m}}\delta_{\alpha_{p,1}, 0}\theta_{1,i}^px_p, \quad 3\leq i \leq k+s.
\end{equation}

From Lie superidentity with $3\leq i \leq k+s$
$$[z_1, [z_2,z_i]]=[[z_1, z_2],z_i] - [[z_1, z_i],z_2],$$
we deduce $[z_1, [z_2,z_i]]\equiv [z_1, \sum\limits_{p=l}^{l_{m}}\theta_{2,i}^px_p] \equiv
\sum\limits_{p=l}^{l_{m}}\theta_{2,i}^p[z_1,x_p]\equiv
-\sum\limits_{p=l}^{l_{m}}\theta_{2,i}^p\alpha_{p,1}x_p.$

On the other hand, we have
$$[[z_1, z_2],z_i] - [[z_1, z_i],z_2]=[\sum\limits_{p=l}^{l_{m}}\delta_{\alpha_{p,1},0}\delta_{\alpha_{p,2},0}\theta_{1,2}^px_p, z_i]-
[\sum\limits_{p=l}^{l_{m}}\delta_{\alpha_{p,1},0}\theta_{1,i}^px_p, z_2]\equiv$$
$$\sum\limits_{p=l}^{l_{m}}\delta_{\alpha_{p,1},0}\delta_{\alpha_{p,2},0}\theta_{1,2}^p[x_p, z_i]-
\sum\limits_{p=l}^{l_{m}}\delta_{\alpha_{p,1},0}\theta_{1,i}^p[x_p, z_2]\equiv
\sum\limits_{p=l}^{l_{m}}\delta_{\alpha_{p,1},0}\Big(\delta_{\alpha_{p,2},0}\theta_{1,2}^p\alpha_{p,i}-\theta_{1,i}^p\alpha_{p,2}\Big)x_p.$$

Let $\alpha_{p,1}=0$ (otherwise we get $[z_1, z_i]=0$ for any $1\leq i \leq k+s$). Then comparing coefficient at the basis elements for any  $l \leq p \leq l_m$ and $3 \leq i\leq k+s$
we derive

\begin{equation}\label{eq37}
\left\{\begin{array}{llllll}
\theta_{1,2}^p\alpha_{p,i}&=&0, &\mbox{if} \quad \alpha_{p,2}=0,\\[1mm]
\theta_{1,i}^p&=&0, & \mbox{if} \quad \alpha_{p,2}\neq 0.\\[1mm]
\end{array}\right.
\end{equation}

If $\alpha_{p,2}=0$, then there exists $j$ such that $\alpha_{p,j}\neq 0$, which imply $[z_1,z_2]=0.$

If
$\alpha_{p,2}\neq 0$, then the product $[z_1,z_2]=0$ follows from \eqref{eq361}. Thus, \eqref{eq37} implies
\begin{equation}\label{eq38}
[z_1, z_2]\equiv0, \quad [z_1, z_i]\equiv\sum\limits_{p=l}^{l_m}\delta_{\alpha_{p,1}, 0}\delta_{\alpha_{p,2}, 0}\theta_{1,i}^px_p, \quad 3\leq i \leq k+s.
\end{equation}

The following congruences for any $r$ ($2 \leq r \leq k+s$):
\begin{equation} \label{eq39}
[z_1,z_i]\equiv 0 , \quad 2\leq i \leq r, \quad [z_1,z_i]\equiv \sum\limits_{p=l}^{l_{m}}\prod\limits_{q=1}^{r}\delta_{\alpha_{p,q},0}\theta_{1,j}^px_p, \quad r+1\leq i \leq k+s,
\end{equation}
we shall prove by induction on $r$.

Thanks to \eqref{eq38} we have the base of induction. Assuming  that congruences \eqref{eq39} are true for $r$ and we shall prove it for $r+1$.

Taking the change $z_{1}'=z_{1}-\sum\limits_{p=l, \alpha_{p,r+1}\neq 0}^{l_m}
\frac{\prod\limits_{q=1}^{r}\delta_{\alpha_{p,q},0}\theta_{1,r+1}^p}{\alpha_{p,r+1}}x_p,$
we can assume
$$[z_1,z_i]\equiv 0 , \quad 1\leq i \leq r, \quad [z_{1},z_{r+1}]\equiv
\sum\limits_{p=l}^{l_{m}}\prod\limits_{q=1}^{r+1}\delta_{\alpha_{p,q},0}\theta_{1, {r+1}}^px_p.$$

Consider now Lie superidentity with $r+2 \leq i \leq k+s$
$$[z_{1}, [z_{r+1},z_i]]=[[z_{1}, z_{r+1}],z_i] - [[z_{1}, z_i],z_{r+1}].$$
Then
$$[z_{1}, [z_{r+1},z_i]]\equiv [z_{1}, \sum\limits_{p=l}^{l_{m}}\theta_{r+1,i}^px_p] \equiv
\sum\limits_{p=l}^{l_{m}}\theta_{r+1, i}^p[z_1,x_p]
\equiv-\sum\limits_{p=l}^{l_{m}}\theta_{r+1,i}^px_p\alpha_{p,1}x_p.$$

On the other hand, we have
$$[[z_{1}, z_{r+1}],z_i] - [[z_{1}, z_i],z_{r+1}]\equiv
\sum\limits_{p=l}^{l_{m}}\prod\limits_{q=1}^{r+1}\delta_{\alpha_{p,q},0}
\theta_{1,{r+1}}^p[x_p, z_i]-
\sum\limits_{p=l}^{l_{m}}\prod\limits_{q=1}^{r}\delta_{\alpha_{p,q},0}\theta_{1,i}^p[x_p, z_{r+1}]\equiv$$
$$\sum\limits_{p=l}^{l_{m}}\prod\limits_{q=1}^{r}\delta_{\alpha_{p,q},0}\Big(\delta_{\alpha_{p,r+1},0}\theta_{1,{r+1}}^p\alpha_{p,i}-
\theta_{1,i}^p\alpha_{p,{r+1}}\Big)x_p.$$

Induction assumption allows us to consider only the case $\alpha_{p,q}=0$ for any $1\leq q \leq r$. Then comparing coefficient at the basis elements we derive
$$\left\{\begin{array}{llllll}
\theta_{1,r+1}^p\alpha_{p,i}&=&0, &\mbox{if} \quad \alpha_{p,r+1}=0,\\[1mm]
\theta_{1,i}^p&=&0, & \mbox{if} \quad \alpha_{p,r+1}\neq 0,\\[1mm]
\end{array}\right.$$
with $l \leq p \leq l_m$ and $r+2\leq i\leq k+s.$

If $\alpha_{p,r+1}=0$, then there exists $i$ such that $\alpha_{p,i}\neq 0$, which imply $[z_1,z_{r+1}]=0$.
Thus, we obtain
$$[z_1,z_i]\equiv 0 , \quad 1\leq i \leq r+1, \quad
[z_1,z_i]\equiv \sum\limits_{p=l}^{l_{m}}\prod\limits_{q=1}^{r+1}\delta_{\alpha_{p,q},0}\theta_{1,i}^px_p, \quad r+2\leq i \leq k+s.$$

Thus, putting $r=k+s$ in \eqref{eq39} we obtain $[z_1,z_i]=0$ for $1\leq i \leq k+s$.

For $2\leq i<j\leq s$ we have $[z_1,[z_i,z_j]]=[[z_1,z_i],z_j]-[[z_1,z_j],z_i]\equiv0.$
On the other hand,
$$[z_1,[z_i,z_j]]\equiv[z_1,\sum\limits_{p=l}^{l_m}\theta_{i,j}^px_p]\equiv-\sum\limits_{p=l}^{l_m}\alpha_{p,1}\theta_{i,j}^px_p.$$
Consequently, $\alpha_{p,1}\theta_{i,j}^p=0,\ l\leq p\leq l_m,\ 2\leq i<j\leq k+s.$
Hence,
\begin{equation}\label{eq310}
[z_1,z_i]\equiv0, \quad 2\leq i\leq k+s,\quad [z_i,z_j]\equiv\sum\limits_{p=l}^{l_m}\delta_{\alpha_{p,1},0}\theta_{i,j}^px_p, \quad
2\leq i< j \leq s.
\end{equation}

Applying induction methods with similar arguments used above and the base of induction congruences \eqref{eq310} one can prove
\begin{equation}\label{eq311}
\left\{\begin{array}{lll}
[z_i,z_j]\equiv0 & 1\leq i\leq r, \ 1\leq j\leq k+s,\\[1mm]
[z_i,z_j]\equiv\sum\limits_{p=l}^{l_m}\prod_{q=1}^{r}\delta_{\alpha_{p,q},0}\theta_{i,j}^px_p, & r+1\leq i \neq j \leq k+s.\\[1mm]
\end{array}\right.
\end{equation}

Setting in congruences \eqref{eq311} $r=k+s$ we obtain $[z_i,z_j]\equiv 0$ for any $1\leq i, j \leq k+s$.
\end{proof}

For a nilpotent Lie superalgebra $\mathcal{N}=\mathcal{N}_0\oplus \mathcal{N}_1$ with a basis $\{x_1, x_2, \dots, x_n, y_1, y_2, \dots, y_m\}$ we set
$$\left\{\begin{array}{lll}
[x_i,x_j]=\sum\limits_{p=1}^{n}\gamma_{i,j}^px_p,& 1\leq i,j\leq n\\[1mm]
[y_i,y_j]=\sum\limits_{p=1}^{n}\delta_{i,j}^px_p,& 1\leq i,j\leq m\\[1mm]
[x_i,y_j]=\sum\limits_{p=1}^{m}\eta_{i,j}^py_p,& 1\leq i \leq n, 1\leq j\leq m\\[1mm]
\end{array}\right.$$
 where $\gamma_{i,j}^t, \delta_{i,j}^t, \eta_{i,j}^t\in \mathbb{C}$ are the structure constants.

Setting in Lemma \ref{lem3.23} $t=\tau+1$ we obtain the main result of this section.
\begin{thm} \label{thm3.24} Let $\cal R=\cal N\oplus \cal Q$ be a maximal solvable extension of an $(n+m)$-dimensional nilpotent Lie superalgebra $\cal N$ of the maximal rank such that $\cal R$ satisfy the conditions:
\

a) $\cal R^2 \subseteq \cal N$;

b) $[\cal R_1, \cal R_1] \subseteq [\cal R_0, \cal R_0]$.\\
Then it admits bases $\{x_1, \dots, x_k, x_{k+1}, \dots, x_{n}\}$ of $\cal N_0$,
$\{y_1, \dots, y_s, y_{s+1}, \dots, y_{m}\}$ of $\cal N_1$, $\{z_1, \dots, z_k, z_{k+1}, \dots, z_{k+s}\}$ of $\cal Q$
with only non-zero additional products to $[\cal N, \cal N]$:

$$\begin{array}{llll}
[x_i,z_j]=\delta_{i,j}x_i, & 1\leq i \leq k, & 1\leq j \leq k+s,\\[1mm]
[x_i,z_j]=\alpha_{i,j}x_i, & k+1\leq i \leq n, & 1\leq j \leq k+s,\\[1mm]
[y_i,z_j]=\delta_{i,j-k}y_i, & 1\leq i \leq s, & 1\leq j \leq k+s,\\[1mm]
[y_i,z_j]=\beta_{i,j}y_i, & s+1\leq i\leq m, & 1\leq j\leq k+s,\\[1mm]
[z_i, z_j]=0, & 1\leq i, j \leq k+s.\\[1mm]
\end{array}$$
\end{thm}

It should be noted that the solvable Lie superalgebras in Example \ref{exam} and Remark \ref{rem3.16} show that conditions $a)$ and $b)$ are essential.

Remark that Theorem \ref{thm3.24} generalises the main result of \cite{Khal}.

Finally, Propositions \ref{prop3.12}, \ref{prop3.20} and Theorem \ref{thm3.24} lead to the following result.

\begin{thm} \label{thm3.25} Solvable Lie superalgebra is of maximal rank if and only if it is maximal solvable extension of a nilpotent Lie superalgebra of maximal rank.
\end{thm}

Now applying Theorem \ref{thm3.25} and Theorem 2.2 in \cite{Wang2003} we obtain
\begin{cor} A maximal solvable extension of an nilpotent Lie superalgebra of maximal rank has trivial center and it admits only inner superderivations.
\end{cor}

\section{Some maximal solvable extensions of nilpotent Lie superalgebras}

In this section we obtain some results on maximal solvable extensions of nilpotent Lie superalgebra $\cal N$  which satisfy the conditions $a)$ and $b)$.

For a fixed basis of $\mathcal{N}$ we consider the system of linear equations
$$S_{x,y}: \quad \left\{\begin{array}{lll}
 \alpha_{i}+\alpha_{j}=\alpha_{t},& \mbox{if} & \gamma_{i,j}^t\neq 0,  \\[1mm]
 \beta_{i}+\beta_{j}=\alpha_{t},& \mbox{if} & \delta_{i,j}^t\neq 0,  \\[1mm]
 \alpha_{i}+\beta_{j}=\beta_{t},& \mbox{if} & \eta_{i,j}^t\neq 0,  \\[1mm]
 \end{array} \right.$$
in the variables $\alpha_1, \dots, \alpha_n$ and $\beta_1, \dots, \beta_m$.

Similar as in the paper \cite{Leger1} we denote by $r\{x_1, \dots, x_n, y_1, \dots, y_m\}$ the rank of the system $S_{x,y}$ and denote by $r\{\mathcal{N}\}=min \  r\{x_1, \dots, x_n, y_1, \dots, y_m\}$ as $\{x_1, \dots, x_n, y_1, \dots, y_m\}$ runs over all bases of $\mathcal{N}$.

In the same way as in \cite{Leger1} for a nilpotent Lie superalgebra $\mathcal{N}$ over an algebraically closed field the equality $\dim \mathcal{T}_{max}=\dim \mathcal{N} -r\{\mathcal{N}\}$ can be proved. Indeed, in order to verify whether a diagonal linear transformation $D=\diag(\alpha_1, \dots, \alpha_n, \beta_1, \dots, \beta_m)$ of an nilpotent Lie superalgebra $\mathcal{N}=\mathcal{N}_0\oplus \mathcal{N}_1$ be an even superderivation we need to apply the Leibniz's rule, which implies the restrictions on diagonal elements $\alpha_{i}, \beta_j$. Namely, from the following chain of equalities
$$\sum\limits_{t=1}^{n}\gamma_{i,j}^t\alpha_tx_t=
D_{z,d}(\sum\limits_{t=1}^{n}\gamma_{i,j}^tx_t)=D_{z,d}([x_i,x_j])=
[D_{z,d}(x_i),x_j]+[x_i,D_{z,d}(x_j)]=(\alpha_i+\alpha_j)\sum\limits_{t=1}^{n}\gamma_{i,j}^tx_t,$$
$$\sum\limits_{t=1}^{n}\delta_{i,j}^t\alpha_tx_t=D_{z,d}(\sum\limits_{t=1}^{n}\delta_{i,j}^tx_t)=
D_{z,d}([y_i,y_j])=[D_{z,d}(y_i),y_j]+[y_i,D_{z,d}(y_j)]=
(\beta_i+\beta_j)\sum\limits_{t=1}^{n}\delta_{i,j}^tx_t,$$

$$\sum\limits_{t=1}^{m}\eta_{i,j}^t\beta_ty_t=
D_{z,d}\left(\sum\limits_{t=1}^{m}\eta_{i,j}^ty_t\right)=D_{z,d}([x_i,y_j])=
[D_{z,d}(x_i),y_j]+[x_i,D_{z,d}(y_j)]=(\alpha_i+\beta_j) \sum\limits_{t=1}^{m}\eta_{i,j}^ty_t$$
we derive that diagonal elements $\alpha_{i}, \beta_j$ have to be a solution of the system $S_{x,y}.$

Clearly, diagonal elements of any even superderivations from $\mathcal{T}_{max}$ are  solutions of the system $S_{x,y}.$

Let $\cal T_{max}=\Span\{d_1, \dots, d_q\}$ with
$d_i=\diag(\alpha_{1,i}, \dots, \alpha_{n,i}, \beta_{1,i}, \dots, \beta_{m,i}), \ i=1, \dots, q.$ Then vectors $d_1, \dots, d_q$ forms a basis of fundamental solutions of the system $S_{x,y}.$

If we consider roots decomposition of $\mathcal{N}$ with respect to $\mathcal{T}_{max}$, then the values of the roots $\alpha_1, \dots, \alpha_n, \beta_1, \dots, \beta_m$ on basis elements of $\mathcal{N}$ will be exactly $\alpha_{-,i}$ and $\beta_{-,i} \ i=1, \dots, q$, that is, $d_i(x_l)=\alpha_{l,i}x_l$ and $d_i(y_t)=\alpha_{t,i}y_t.$

%
%

%

\begin{prop}\label{prop3}
Let $\cal N=\cal N_{0}\oplus\cal N_{1}$ be a nilpotent Lie superalgebra of maximal rank such that $[\mathcal{N}_1,\mathcal{N}_1]\subseteq \mathcal{C}^2(\mathcal{N}_0)$. Then, $\cal N$ is a Lie algebra.
\end{prop}

\begin{proof} Let $\{x_1, \dots, x_k\}$ and $\{y_1, \dots, y_s\}$ are even and odd generator basis elements of $\mathcal{N}$, respectively. From conditions of proposition we have
$$\dim \mathcal{T}_{max}=
\dim (\mathcal{N}_0/{\cal C}^{2}(\mathcal{N}_0)) + \dim (\mathcal{N}_1/{\cal C}^{2}(\mathcal{N}_1)).$$

A general element $d$ of $\mathcal{T}_{max}$ has the following diagonal form:\
$$d=\diag(\alpha_1, \dots, \alpha_k, \alpha_{k+1}, \dots, \alpha_n, \beta_1, \dots, \beta_s, \beta_{s+1}, \dots, \beta_m),$$
where $\alpha_i, \ 1\leq i \leq k$, $\beta_j, \ 1 \leq j \leq s$ are free parameters and the rest of the parameters are linearly dependent on these parameters in the system $S_{x,y}.$

Clearly, $\dim {\mathcal{T}_{max}}_{|\cal N_0}=k.$ Consider now a solvable Lie algebra ${\cal R_{\mathcal{T}_{max}}}_{|\mathcal{N}_0}=\cal N_0\dot{+}{\mathcal{T}_{max}}_{| \mathcal{N}_0}.$ Then due to results of \cite{Khal} and \cite{Meng} we have the existence of a basis $\{z_1,\dots,z_k\}$ of ${\mathcal{T}_{max}}_{|\cal N_0}$ such that
\begin{align*}
[x_i,z_i]&=x_i, \quad 1\leq i\leq k,\\
[x_i,z_j]&=\alpha_{ij}x_j, \quad k+1\leq i\leq n, \ 1\leq j\leq k,
\end{align*}
where $\alpha_{ij}$ is the number of entries of a generators basis element $x_j$ involved in forming of non-generator basis element $x_i$. This implies that for any $l\in \{1, \dots, n\}$ there exists $\bar{d}\in \mathcal{T}_{max}$ such that $\bar{d}(x_l)=\alpha_lx_l\neq 0.$

Let us assume that $[\mathcal{N}_1,\mathcal{N}_1]\neq 0$. Then there exist $i_0, \ j_0$ and $t_0$ such that $[y_{i_0}, y_{j_0}]=\sum\limits_{l=t_0}^{n}\gamma_l x_l$ with $\gamma_{t_0}\neq 0$. For $d\in \mathcal{T}_{max}$ with $d(x_{t_0})=\alpha_{t_0}x_{t_0}\neq 0$ we have
$$d([y_{i_0}, y_{j_0}])=[d(y_{i_0}), y_{j_0}]+[y_{i_0}, d(y_{j_0})]=(\beta_{i_0}+\beta_{j_0})[y_{i_0}, y_{j_0}]=
(\beta_{i_0}+\beta_{j_0})\left(\sum\limits_{l=t_0}^{n}\gamma_l x_l\right).$$
On the other hand, we have
$$d([y_{i_0}, y_{j_0}])=d\left(\sum\limits_{l=t_0}^{n}\gamma_l x_l\right)=\sum\limits_{l=t_0}^{n}\alpha_l\gamma_lx_l.$$

Comparing coefficient at the basis element $x_{t_0}$, we get $\beta_{i_0}+\beta_{j_0}=\alpha_{t_0},$ which implies a relation between free parameters $\alpha_1, \dots, \alpha_k$ and $\beta_1, \dots, \beta_s.$ Consequently,
$\dim \mathcal{T}_{max}<k+s$, this is a contradiction with maximality of rank of $\mathcal{N}$. Therefore, we obtain $[\mathcal{N}_1, \mathcal{N}_1]=0,$ which complete the proof of proposition.
\end{proof}


%

Now we present an example which shows that maximal solvable extension of an nilpotent Lie superalgebra does not guarantee the maximality of solvable extension of its even part.

\begin{exam} \label{exam3.17} Let us consider maximal solvable extension of the nilpotent Lie superalgebra $\cal N^2$ in Example \ref{exam111}. Then it is isomorphic to the solvable Lie superalgebra with the table of multiplications:
$$\left\{\begin{array}{lll}
[\cal N^2, \cal N^2], &\\[1mm]
[x_1,z_1]=x_1,&[x_2,z_2]=2x_2,\\[1mm]
[y_1,z_1]=-y_1,&[y_1,z_2]=y_1,\\[1mm]
[y_3,z_1]=y_3,&[y_2,z_2]=y_2,\\[1mm]
[y_3,z_2]=y_3.
\end{array}\right.$$

Clearly, $\cal N^2_0\oplus \cal Q_0^1$ is three-dimensional solvable Lie algebra with two-dimensional abelian nilradical. However, it is well-known that maximal solvable extension of two-dimensional abelian algebra is four-dimensional.
\end{exam}




%


Consider the roots decomposition of odd part of the nilradical $\mathcal{N}$ with respect to $\mathcal{T}_{max}.$ Namely,
$$\begin{array}{ll}
\mathcal{N}_1=\mathcal{N}_{\beta_1}\oplus\dots \oplus \mathcal{N}_{\beta_s}\oplus\mathcal{N}_{\beta_{s+1}}\dots
\oplus \mathcal{N}_{\beta_v}.\\[1mm]
\end{array}$$


\begin{thm} \label{thm3.30}
Let $\cal R=\cal R_0\oplus \cal R_1$ be a complex maximal solvable extension of $\cal N=\cal N_0 \oplus \cal N_1$ which satisfies the following conditions:

i) \quad  $\cal R_0$ is solvable Lie algebra of maximal rank;

ii) \quad $\beta_i\neq \beta_j$ for any $i\in\{1 \dots, s\}$ and $j\in\{1, \dots, v\}$.

 Then $\cal R\cong \mathcal{R}_{\mathcal{T}_{max}}.$
\end{thm}
\begin{proof} Due to the condition $i)$ we have that $\mathcal{R}_0=R_{\widehat{\mathcal{T}}_{max}}$, where $\widehat{\mathcal{T}}_{max}$ is a maximal torus of $\mathcal{N}_0$. Applying the main result of \cite{Khal} we have the existence of a basis $\{z_1,\dots,z_{k}\}$ (here $k=\dim \widehat{\mathcal{T}}_{max}$) such that
\begin{align}\label{qn}
[x_i,z_i]=x_{i}, \quad [x_j,z_{i}]=\alpha_{j,i}x_{j}, \quad i=1,\dots,k, \ j=1,\dots, n
\end{align}
where $\alpha_{j,i}$ forms a fundamental solution of the subsystem $S_x$ of the system $S_{x,y}$.

Taking into account results of the work \cite{Meng} we get the equality
$[\mathcal{N}_0, \widehat{\mathcal{T}}_{max}]=\mathcal{N}_0$, which implies
$$[\mathcal{R}_1,\mathcal{R}_1]=[\mathcal{N}_1, \mathcal{N}_1] \subseteq \mathcal{N}_0=[\mathcal{N}_0, \widehat{\mathcal{T}}_{max}]\subseteq [\mathcal{R}_0, \mathcal{R}_0].$$
Thus, due to these embeddings we conclude that analogue of Lie's theorem for the solvable Lie superalgebra  $\mathcal{R}$ is true. So, we only need to describe the product $[\cal N_1,\widehat{\mathcal{T}}_{max}]$ for which is sufficient to describe it on generator elements of $\cal N_1.$

Thanks to analogue of Lie's theorem one can introduce notations
$$[y_i,z_j]=\beta_{i,j}y_i+\sum_{p=i+1}^{m}\mu_{i,j}^{p}y_p, \quad i=1, \dots, s, \quad j=1, \dots, k.$$

Applying equality $[z_j,z_t]=0$ in Leibniz superidentity for the triple of elements $\{y_i, z_j, z_t\}$ with $1\leq i \leq s$ and
$1\leq j, t \leq k,$ we derive
$$0=\Big(\beta_{i,t}\mu_{i,j}^{i+1}+\beta_{i+1,j}\mu_{i,t}^{i+1}
-\beta_{i,j}\mu_{i,t}^{i+1}-\beta_{i+1,t}\mu_{i,j}^{i+1}\Big)y_{i+1}+
\sum_{p=i+2}^{m}(*)y_{p}.$$

Consequently, we obtain
\begin{equation}{\label{ijk}}
(\beta_{i,t}-\beta_{i+1,t})\mu_{i,j}^{i+1}=(\beta_{i,j}-\beta_{i+1,j})\mu_{i,t}^{i+1}.
\end{equation}
Due to condition $ii)$ we conclude that for any $r, l\ (1\leq r\neq l \leq s)$ there exists $t\in\{1,\dots,k\}$ such that $\beta_{r,t}\neq\beta_{l,t}.$ We choose $t\in\{1,\dots,k\}$ such that $\beta_{i,t}\neq\beta_{i+1,t}.$

Let us assume that $\beta_{i,j}\neq\beta_{i+1,j}$ (otherwise we get $\mu_{i,j}^{i+1}=0$). Then Equality  \eqref{ijk} implies
$\mu_{i,j}^{i+1}=\frac{\beta_{i,j}-\beta_{i+1,j}}{\beta_{i,t}-\beta_{i+1,t}}\mu_{i,t}^{i+1}$.

Taking a basis change as follows: $y_{i}^{\prime}=y_i+\frac{\mu_{i,j}^{i+1}}{\beta_{i,j}-\beta_{i+1,j}}y_{i+1},$ one can assume
$$[y_i,z_j]=\beta_{i,j}y_i+\sum\limits_{p=i+2}^{m}\mu_{i,j}^{p}y_p, \quad i=1, \dots, s, \quad  j=1,\dots,k.$$

Sequentially applying the same arguments $s$-times, finally we get
$$[y_i,z_j]\equiv\beta_{i,j}y_i, \quad i=1, \dots, s, \quad j=1, \dots, k \quad  (\mbox{mod} \ \mathcal{N}^2).$$

Now applying the condition $ii)$ and arguments similar as in the proof of Lemma \ref{lem3.23} we deduce
$$[y_i,z_j]\equiv\beta_{i,j}y_i \quad i=1, \dots, s, \ j=1, \dots, q \quad  (\mbox{mod} \ \mathcal{N}^{t+1}).$$

Continuing this process until the nilindex of $\cal N$, finally, we obtain
\begin{equation}\label{eq3.15}
[y_i,z_j]=\beta_{i,j}y_i, \quad i=1, \dots, s, \ j=1, \dots, q.
\end{equation}

Thus, the products (\ref{qn}) and (\ref{eq3.15}) provide the existence of an isomorphism $\varphi$ such that $\varphi(\cal R)=\cal R_{\mathcal{T}_{max}}.$
\end{proof}

The following solvable Lie superalgebra with the non-zero products
$$[y_1,y_2]=x_1, \quad [x_1,z_1]=x_1, \quad [y_1,z_1]=y_1, \quad
[x_1,z_2]=x_1, \quad  [y_2,z_2]=y_2,$$
satisfy the conditions $i)-ii)$ of Theorem \ref{thm3.30}.

\end{document}